\documentclass[12pt,a4paper,twoside]{article}

\usepackage[latin1]{inputenc}
\usepackage[brazil,english]{babel}
\usepackage{amsfonts}
\usepackage{amssymb}
\usepackage{amsmath}
\usepackage{amsthm}
\usepackage[usenames,dvipsnames]{color}
\usepackage{enumitem}
\usepackage[normalem]{ulem}
\usepackage{setspace}

\numberwithin{equation}{section}
\newtheorem{theorem}{Theorem}[section]
\newtheorem*{theorem*}{Theorem}

\newtheorem{lemma}[theorem]{Lemma}
\newtheorem*{lemma*}{Lemma}

\newtheorem{example}[theorem]{Example}
\newtheorem{remark}[theorem]{Remark}

\makeatletter
\newcommand{\aslabel}[1]{#1\def\@currentlabel{#1}}
\newcommand{\nowlabel}[1]{\def\@currentlabel{#1}}
\makeatother

\newcommand{\N}{{\mathbb{N}}}

\newcommand{\hsp}{\hspace*{.5cm}}

\usepackage{textcomp,amssymb,wasysym}

\usepackage{ulem} 

\usepackage{amsmath}

\definecolor{collnk}{rgb}{.2,0.2,.6}
\definecolor{colcit}{rgb}{.1,0.5,.2}
\usepackage[colorlinks=true]{hyperref}\hypersetup{urlcolor=blue, citecolor=colcit,linkcolor=collnk}

\RequirePackage[normalem]{ulem} 
\RequirePackage{color}\definecolor{RED}{rgb}{1,0,0}\definecolor{BLUE}{rgb}{0,0,1} 

\begin{document}
\title{Super-exponential decay rates for eigenvalues and singular values of integral operators on the sphere}
\author{Mario H. Castro$^a$\footnote{Partially supported by FAPEMIG/CNPq grant \# CEX - APQ-00474-14}, \, Tha\'is Jord\~ao$^b$\footnote{Partially supported by FAPESP grants \# 2016/02847-9 and \# 017/07442-0} \, \& \, Ana P. Peron$^b$\footnote{Partially supported by FAPESP grant  $\#$2016/09906-0}
\\ \scriptsize{$^a$Departamento de Matem\'{a}tica, UFU,  Uberlândia - MG, Brasil. 
}\\\scriptsize{$^b$Departamento de Matem\'{a}tica, ICMC - USP, S\~{a}o Carlos - SP, Brasil.}\\ \scriptsize{
		e-mails: mariocastro@ufu.br,\,\,\,\, tjordao@icmc.usp.br, \,\,\,\, apperon@icmc.usp.br}} 

\maketitle 
\begin{abstract}
This paper brings results about the behavior of sequences of eigenvalues or singular values of integral operators generated by square-integrable kernels on the real $m$-dimensional unit sphere, $m\geq2$. Under smoothness assumptions on the generating kernels, given via Laplace-Beltrami differentiability, we obtain super-exponential decay rates for the eigenvalues of the  generated positive integral operators and for singular values of those integral operators which are non-positive. We show an optimal-type result  and provide a list of parametric families of kernels which are of interest for numerical analysis and geostatistical communities and satisfy the smoothness assumptions for the positive case.\\

\\

\end{abstract}

\section{Introduction}

The relation between the smoothness of generating kernels and decay rates for eigenvalues or singular values of the generated integral operators was largely explored in the last century and can be useful in branches of mathematics and applied mathematics as it can be seen in \cite{wa-zhu} and references quoted there. It first appeared in Fredholm's fundamental paper \cite{fred} and it has been developed since then by many authors in different settings with several smoothness assumptions. 

There is a considerable amount of results on this subject in the literature, mostly considering Lipschitz and Hölder type assumptions (\cite{bir-sol}), but differentiability also yields useful smoothness conditions (see \cite{chang-ha99, han90, fe-men-pe, jor-men-sun} and references therein). Besides, it is known that finite orders of differentiability produce polynomial type decay rates for the eigenvalues or singular values sequences. On the other hand, it was shown in \cite{reade} that positive sequences decreasing to zero at least faster than $\{n^{-p}\}_n$ are related to $p$-times continuously differentiable kernels. As so we may expect some kind of exponential rates for the decrease of these sequences when the order of smoothness of the kernel is assumed to be infinite.

In fact, this is what happens e.g. in \cite{raman-rao, parfenov, little-reade, aze-men-sharp}  and, in particular, in \cite{kotljar} where B. D. Kotlyar  considered Hilbert-Schmidt integral operators  in $L^2(a,b)$ generated by complex kernels having mean derivatives  of all orders increasing sufficiently rapidly and achieved for the singular values $s_n$  the estimate 
$$
s_n\leq \sqrt{M}\, n^{2+\log_2((b-a)/\pi)-\log_2\sqrt{(2/R)}\log_2 n},
$$
where $R<2$ and $M$ are positive numbers not depending on $n$ but both connected with the smoothness assumption.

In this paper we follow some ideas of \cite{cas-men} and present spherical versions of  Kotlyar's result quoted above. The paper is organized as follows. We state our main results in Section \ref{sec_main-results} after fixing some notations and providing the necessary definitions of integral operators on the spherical context as well as the concept of smoothness via Laplace-Beltrami derivatives. In Section \ref{sec_backg-sph} we discuss some spherical results to be used later. In Section \ref{sec_proofmain} we prove the main theorems of this paper. In Section \ref{sec_optimal} we present an example providing an optimality-type result. Section \ref{sec-examples} presents a list of parametric families of kernels which fit Theorem \ref{maintheorem2} and are important for several applications.
In Section \ref{appendix} we explain how to obtain the main results in the setting of compact two-point homogeneous spaces.

\section{Main results} \label{sec_main-results}

Let $m\geq 2$ be an integer fixed throughout this text and $S^m$ the $m$-dimensional unit sphere centered at the origin of the Euclidean space $\mathbb{R}^{m+1}$. We endow $S^m$ with the surface measure $\sigma_m$ induced by the Lebesgue measure of $\mathbb{R}^{m+1}$ and consider the Hilbert space $L^2(S^{m}):=L^2(S^{m},\sigma_m)$  consisting of all square integrable functions 
$f:S^m\to\mathbb{C}$ with the norm $\|\cdot\|_2$ induced by the inner product
\begin{equation}
\langle f, g \rangle_2=\frac{1}{\omega_m}\int_{S^{m}}f(x)\overline{g(x)}d\sigma_m(x), \quad f,g\in L^2(S^m),
\end{equation}
where $\omega_m$ means the volume of $S^{m}$. 

We call {\em kernel } on $S^m$ a function $K:S^m\times S^m\to \mathbb{C}$ and as before we consider the Hilbert  space $ L^2(S^{m}\times S^m)=:L^2(S^{m}\times S^m,\sigma_m\otimes\sigma_m)$.  It is well known that $K\in L^2(S^{m}\times S^m)$ generates a compact integral operator  $\mathcal K:L^2(S^m)\to L^2(S^m)$ defined by
\begin{equation}
\mathcal{K}(f)=\int_{S^m}K(\cdot,y)f(y)d\sigma_m(y).
\end{equation}

The set  of eigenvalues of $\mathcal{K}$ is known to be a discrete subset $\{\lambda_n(\mathcal{K})\}$ of complex numbers having only zero as limit point. The eigenvalues of the compact operator $|\mathcal{K}|=:(\mathcal{K}^*\mathcal{K})^{1/2}$, in which $\mathcal{K}^*$ stands for the adjoint operator of $\mathcal{K}$, are denoted by $s_n(\mathcal{K})$ and called the \emph{singular values} of $\mathcal{K}$. Since $|\mathcal{K}|$ is also positive and self-adjoint then  the spectral theorem can be invoked in order to ensure that $\{s_n(\mathcal{K})\}$ is real and can be distributed as $$s_1(\mathcal{K})\geq s_2(\mathcal{K})\geq\dots\geq0,$$ regarding possible repetitions caused by the algebraic multiplicities.
 
When $\mathcal{K}$ is also positive we say that $K$ is {\it $L^2$-positive definite} ($L^2$-PD for short).
This implies that $\mathcal{K}$ is also self-adjoint and its eigenvalues  $\lambda_n(\mathcal{K})$  are non-negative real 
numbers to be arranged in a non-increasing order. If $K$ is also continuous then it becomes positive definite on $S^m$  in the sense that
$$
 \sum _{i=1}^{n}\sum _{j=1}^{n}c_{i}\overline{c_{j}}K(x_{i},x_{j})\geq 0,
$$ 
holds for any $ n\in \mathbb {N}$, $x_{1},\dots ,x_{n}\in S^m$, $c_{1},\dots ,c_{n}\in \mathbb{C}$ (see Theorem 2.3 in \cite{fe-men}). 

The smoothness assumption we adopt is given via Laplace-Beltrami derivatives, a variation of the usual derivative on $S^m$. Consider the {\em spherical shifting}  defined by the formula
\begin{equation}
S_{t}^{m}(f)(x):=\frac{1}{|R_t^m|}\int_{
x\cdot y=t}f(y)\,dy, \quad x \in S^{m}, \quad t \in (-1,1),
\end{equation}
where ``$x\cdot y$" is the usual inner product among $x,y\in \mathbb{R}^{{m+1}}$, $dy$ denotes the measure element of the rim $R_t^m:=\{y \in S^{m}: x\cdot y=t\}$ of $S^m$ and $|R_t^m|=\omega_{m-1}(1-{t}^{2})^{(m-1)/2}$ its volume.  We say 
$f \in L^2(S^m)$ is  {\em LB-differentiable} if there exists $\mathcal{D}f \in L^2(S^m)$ such that
\begin{equation}
\lim_{t\to 1^-}\left\|\frac{(I-S_{t}^{m})(f)}{1-t}-\mathcal{D}f\right\|_2=0,
\end{equation}
where $I$ is the identity operator. That being true, the function $\mathcal{D}f$ is called the {\em Laplace-Beltrami derivative} of $f$.\ Higher order derivatives are inductively defined by $\mathcal{D}^1=\mathcal{D}$ and $\mathcal{D}^r:=\mathcal{D}^{1}\circ \mathcal{D}^{r-1}$, $r=2,3,\dots$. We say $f\in L^2(S^m)$ is {\em infinitely LB-differentiable} when it has Laplace-Beltrami derivatives of all orders. 

Properties of Laplace-Beltrami derivatives we will need in order to prove our results will be presented in the next section.\ For more information about this subject we suggest \cite{men-pia} and references therein.\ In particular, the reader can find in \cite{ca-men-oliv} the connection among the Laplace-Beltrami derivative and the usual derivative on $S^m$.

Basic Sobolev-type spaces can be defined following \cite[p.37]{lions-ma}.\ The space of all complex valued functions on $S^m$ which are LB-differentiable up to order $r$ is denoted by $W_2^r$ while the space of those functions that are infinitely LB-differentiable is defined by  
\begin{equation}
W_2^{\infty}:=\bigcap_rW_2^r=\{f\in L^2(S^m): \mathcal{D}^rf\in L^2(S^m), \quad r=1,2,\dots\}.
\end{equation}

 The action of  Laplace-Beltrami derivatives on 
kernels on $S^m$ is similar to that of partial (usual) derivatives: we write $\mathcal{D}_y^rK$  to indicate the $r$-th order Laplace-Beltrami derivative of $K$ with respect to the variable  $y$. In particular, we write
\begin{equation}
K_{0,r}:= \mathcal{D}_y^rK, \quad r=1,2,\ldots,
\end{equation}
 and denote by $\mathcal{K}_{0,r}$ the integral operator generated.   Finally, we say that $K$ {\em belongs to} $W_2^{\infty}$ whenever 
\begin{equation}
 K(x,\cdot)\in W_2^{\infty}, \quad x\in S^m \quad \text{a.e.}.
\end{equation}

Now we are ready to state the main results of this paper.

\begin{theorem}\label{maintheorem}  Let $K\in L^2(S^m\times S^m)$ belong to $W_2^{\infty}$.\ If 
	there exist constants $M>0$ and $R>1$, do not depending upon $r$, such that
	\begin{equation}\label{normbound}
	\|K_{0,r}\|_2 \leq MR^r,
	\end{equation}
then, for all positive sequences $\alpha_n=o(\sqrt[m]{n})$, as $n\to\infty$,
\begin{equation}\label{eq-decay-sing-v}
s_n(\mathcal{K})=o\left(n^{-{\frac{\sqrt[m]{n}}{\delta m} - \alpha_n}}\right),\quad 
\delta=\begin{cases}
2,& \text{if } m=2,\\
1,& \text{if } m\geq3.
\end{cases}
\end{equation}
\end{theorem}

This result can be seen as a spherical version of the one obtained on bounded 
intervals by B. D. Kotlyar in \cite{kotljar}. Similar to  its original real version, it applies to compact integral operators generated by kernels having all orders of Laplace-Beltrami derivatives rapidly increasing. But verifying that a kernel satisfies condition \eqref{normbound} can be a quite difficult task. As so we replace \eqref{normbound} by a positivity assumption on the kernel, an easier condition to be verified, though some kind of boundedness of   $\mathcal{K}_{0,r}$ 
 must remain in order to ensure technical properties of the singular values.

\begin{theorem}\label{maintheorem2} Let $K$ be $L^2$-PD and belong to $W_2^{\infty}$.\ If 
	 $\mathcal{K}_{0,r}:L^2(S^m)\to L^2(S^m)$ is bounded for all $r\in\mathbb{N}$
then, for all positive sequences $\alpha_n=o(\sqrt[m]{n})$, as $n\to\infty$,
\begin{equation}\label{eq-decay-eig}
	\lambda_n(\mathcal{K})=o\left(n^{-{\frac{\sqrt[m]{n}}{\delta m} - \alpha_n}}\right),\quad
 \delta=\begin{cases}
2,& \text{if } m=2,\\
1,& \text{if } m\geq3.
\end{cases}
\end{equation}
\end{theorem}

\begin{remark}
The decay rates \eqref{eq-decay-sing-v} and \eqref{eq-decay-eig} are obtained after showing the convergence of the series
\begin{equation}\label{serieseigenvalues}
\sum_{n=1}^{\infty}n^{{\sqrt[m]{n}\over\delta m}+{\alpha_n}}a_n,\quad
 \delta=\begin{cases}
2,& \text{if } m=2,\\
1,& \text{if } m\geq3,
\end{cases}
\end{equation}
where $a_n$ denotes $s_n(\mathcal{K})$ in the proof of Theorem \ref{maintheorem}, and $\lambda_n(\mathcal{K})$ in the the proof of Theorem \ref{maintheorem2}. Stronger and more explicit estimates for particular subsequences of the singular values of $\mathcal{K}$ follow from Lemma \ref{lemmavalorsingular} and also substantiate the convergence of the series \eqref{serieseigenvalues} (see Equations \eqref{estimatives2n1}, \eqref{step0}  and \eqref{estimatives2n2}).
\end{remark}

\begin{remark}\label{noteusualderivative} Both previous results remain true if we consider usual derivatives instead of the smoothness assumption used. This occurs because every function having all orders of derivatives (in the usual meaning) belongs to $W_2^{\infty}$. We refer \cite[p.102]{ca-men-oliv} to see how to go   from usual derivative to 
Laplace-Beltrami derivative.
\end{remark}

\begin{remark} \label{rem-homog}
 After some adjustments, both Theorems \ref{maintheorem} and \ref{maintheorem2} can be reproduced for compact two-point homogeneous spaces  of dimension $m\geq2$ .\ These spaces are  compact  manifolds having an invariant Riemannian (geodesic) metric and can be endowed with a measure that permits to consider the Laplace-Beltrami operator on it, which can be used in order to define the smoothness condition in the same way that in the spherical case.\  We give more details in Section \ref{appendix} .
\end{remark}

\section{Spherical background}\label{sec_backg-sph}

In this section we review several properties related to the Laplace-Beltrami derivative operator $\mathcal{D}^r\colon W_2^r\to L^2(S^m)$ and the Laplace-Beltrami integral operator.

The powers of the Laplace-Beltrami integral operator appear quite
naturally in decompositions of $\mathcal{K}$ when the generating
kernel $K$ satisfies smoothness assumptions defined via the
Laplace-Beltrami derivative (see Lemma 4.3 in \cite{cas-men} and Lemma \ref{lemmafactorization} in this section).\ For that reason, the Laplace-Beltrami integral operator enters in the proofs of the main results
previously stated.

We denote by $\mathcal{H}_n^{m+1}$ the class of all spherical harmonics of degree $n$ in $m+1$ variables. The set of all spherical harmonics in $m+1$ variables constitutes a fundamental set of $L^2(S^{m})$, that is, $L^2(S^{m})=\oplus_{n=0}^{\infty}\mathcal{H}_n^{m+1}$.  The dimension of $\mathcal{H}_n^{m+1}$ will be written as $d_n^m$ while $\{Y_{n,k}: k=1,2,\ldots, d_n^{m}\}$ will stand for an orthonormal basis of it. The numbers $d_n^m=\dim
\mathcal{H}_n^{m+1}$ are given by $d_0^m=1$ and  (\cite[p.17]{mori})
\begin{equation}\label{dnm}
d_n^m=(2n+m-1){(n+m-2)!\over n!m!},\quad n\geq 1.
\end{equation} 
They can also be obtained via recurrence relation (\cite[p.18]{mori}) as
\begin{equation}\label{recurrence}
d_n^{m+1}=\sum_{k=0}^n{d_k^{m}}.
\end{equation} 
The $r$-th order Laplace-Beltrami derivative operator, $ \mathcal{D}^r\colon W_2^r\to L^2(S^m)$ is an unbounded self-adjoint operator acting as a multiplier in the sense that  
\begin{equation}\label{eigenfunctionharmonic}
\mathcal{D}^r(Y)={n^r(n+m-1)^r\over m^r}Y, \quad Y\in \mathcal{H}_n^{m+1}.
\end{equation}

The  {\em Laplace-Beltrami integral operator} is the unique linear
mapping $J\colon L^2(S^m)$ $\rightarrow L^2(S^m)$ defined by the
conditions $J(1)=1$ and
\begin{equation}\label{eq3.1}J(Y):=\frac{m}{n(n+m-1)}\,Y, \quad Y \in \mathcal{H}_n^{m+1},\quad n=1,2,\ldots.\end{equation}
It is a bounded linear operator acting like an inverse of the
Laplace-Beltrami derivative operator in the sense that
\begin{equation}\label{eq3.2}
\mathcal{D}( J (Y))=J (\mathcal{D} (Y))=Y, \quad Y \in
\oplus_{n=1}^{\infty}\mathcal{H}_n^{m+1}.
\end{equation}  
It can also be defined via spherical convolution 
as the reader can see in \cite{men-pia}.
The powers of $J$ are also defined recursively:
$J^1=J$ and $J^r:=J\circ J^{r-1}$, $r=2,3,\ldots$.

The formula
\begin{equation}\label{eq3.6}\langle J^r (f), \overline{g} \rangle_2 = \langle f, \overline{J^r (g)}\rangle_2,\quad  f,g \in L^2(S^m),\end{equation}
encompasses the self-adjointness of $J^r$ while Theorem 3.1 in \cite{cas-men} shows the operator $J^r$ is also compact. 
Due to the spectral theorem, its eigenvalues can be listed in a
 non-increasing order counting the repetitions according to the formulas
 $J^r\,1=1$ and \eqref{eq3.1}. Therefore  the sequence
 $\{\lambda_n(J^r)\}_n$ can be block-ordered such that:
 \begin{itemize}
 	\item[-] the first block contains the
 	eigenvalue 1; and
 	\item[-] the $(n+1)$-th block ($n\geq 1$) contains $d_n^m$
 	entries equal to $m^{r}n^{-r}(n+m-1)^{-r}$.
 \end{itemize}
 Consequently, by \eqref{recurrence}, the first element in the $(n+1)$-th block corresponds to the index
 \begin{equation}\label{eq-primeiro-eig-bloco-n+1}
 d_0^m+d_1^m+\cdots +d_{n-1}^m+1=d_{n-1}^{m+1}+1,
 \end{equation}
 and the last one corresponds to the index
 \begin{equation}\label{index}
 d_0^m+d_1^m+\cdots +d_{n-1}^m+d_n^m=d_n^{m+1}.
 \end{equation}

To close the section we state the following result concerning an estimation for the singular values of $\mathcal{K}$. It plays a crucial role in the proofs of Theorems \ref{maintheorem} and \ref{maintheorem2}.
We omit the proof as it can be seen as a corollary of Lemma 4.3 in \cite{cas-men}.

\begin{lemma}\label{lemmafactorization}Let $K$ be an element of $W_2^{\infty}$.\ If
$\mathcal{K}_{0,r}:L^2(S^m)\to L^2(S^m)$ is bounded for all $r$ then
\begin{equation*}s_{n+1}(\mathcal{K})\leq s_n(\mathcal{K}_{0,r}J^r),\quad r,n=1,2,\dots.\end{equation*}\end{lemma}

\section{Proofs of the main results}\label{sec_proofmain}

Before proceeding to the proofs of our main results, we review one important estimate and two identities concerning the singular values of compact operators. All of them come from \cite{konig, pietsch}.

If $T$ and $A$ are compact operators on $L^2(S^m)$ then 
\begin{equation}\label{lemma sing val2}
s_{n+k-1}(AT) \leq s_{n}(A)s_{k}(T),\quad n,k\in\mathbb{N}.
\end{equation}
Moreover, if $T$ is self-adjoint then its singular values are related to its eigenvalues by
\begin{equation}\label{lemma sing val1}
s_{n}(T)=|\lambda_{n}(T)|,\quad n\in\mathbb{N}.
\end{equation}
Furthermore, if $K\in L^2(S^m\times S^m)$ then the following relation involving the singular values of the integral operator generated by $K$ holds true:
\begin{equation}\label{normL2singvalues}\sum_{n=1}^{\infty}s_{n}^2(\mathcal{K})=\|K\|^2_{2}.\end{equation}

The next lemma plays a foundation role for the proofs of Theorems \ref{maintheorem} and \ref{maintheorem2}.

\begin{lemma}\label{lemmavalorsingular}
Let $K\in L^2(S^m\times S^m)$ belong to $W_2^{\infty}$. If $\mathcal{K}_{0,r}:L^2(S^m)\to L^2(S^m)$ is bounded for all $r\in\mathbb{N}$ then there exists a constant $c>0$ satisfying
\begin{equation}\label{estimatives2n}
s_{d_n^{m+1}}(\mathcal{K})\leq c~{e^{2n}m^{n} \over n^{2n+m}}~s_1(\mathcal{K}_{0,n}),\quad n\in\mathbb{N}.
\end{equation}
\end{lemma}
\begin{proof}[{\bf Proof. }] 
Suppose $\mathcal{K}_{0,r}:L^2(S^m)\to L^2(S^m)$ is bounded for all $r\in\mathbb{N}$. Lemma \ref{lemmafactorization} yields
$$
s_{d_n^{m+1}}(\mathcal{K})\leq s_{d_n^{m+1}-1}(\mathcal{K}_{0,1}J), \quad n\in\mathbb{N}.
$$
The recurrence formula \eqref{recurrence} leads to
\begin{equation}\label{step00}
s_{d_n^{m+1}}(\mathcal{K})\leq s_{d_n^{m}+d_{n-1}^{m+1}-1}(\mathcal{K}_{0,1}J),   \quad n\in\mathbb{N},
\end{equation} 
while inequality \eqref{lemma sing val2} guarantees that
\begin{equation}\label{step000}
s_{d_n^{m}+d_{n-1}^{m+1}-1}(\mathcal{K}_{0,1}J)\leq  s_{d_n^{m}}(J)s_{d_{n-1}^{m+1}}(\mathcal{K}_{0,1}), \quad n\in\mathbb{N}.
\end{equation}
Inequalities
 \eqref{step00} and \eqref{step000} imply
\begin{equation}\label{step1}
s_{d_n^{m+1}}(\mathcal{K})\leq  s_{d_n^{m}}(J)s_{d_{n-1}^{m+1}}(\mathcal{K}_{0,1}) ,  \quad n\in\mathbb{N}.
\end{equation}
The  same  reasoning applied for $s_{d_{n-1}^{m+1}}(\mathcal{K}_{0,1})$ helps us to reach
$$
s_{d_{n-1}^{m+1}}(\mathcal{K}_{0,1})\leq s_{d_{n-1}^{m}}(J)s_{d_{n-2}^{m+1}}(\mathcal{K}_{0,2}),  \quad n=2,3,\dots.
$$
After replacing this very last inequality in \eqref{step1} we are conducted  to
\begin{equation*}\label{step2}
s_{d_n^{m+1}}(\mathcal{K})\leq  s_{d_n^{m}}(J)s_{d_{n-1}^{m}}(J)s_{d_{n-2}^{m+1}}(\mathcal{K}_{0,2}) ,  \quad n=2,3,\dots.
\end{equation*}
The smoothness assumption on $K$ and a recursive process allow us to write
\begin{eqnarray}\label{step3}
s_{d_n^{m+1}}(\mathcal{K})\leq  \left[\prod_{i=1}^n s_{d_i^{m}}(J)\right]s_1(\mathcal{K}_{0,n}) ,\quad n\in\mathbb{N}.
\end{eqnarray}
\\
We turn our attention to the newcomer product on the right side of \eqref{step3}.\ Since we know  
the singular values of the Laplace-Beltrami integral operator, we are able to establish  
\begin{eqnarray}\label{step4}
\prod_{i=1}^n s_{d_i^{m}}(J)= \prod_{i=1}^n {m\over i(i+m-1)} = {m^n(m-1)! \over n!\,(n+m-1)!}, \quad n\in\mathbb{N}.
\end{eqnarray}
Stirling's formula (\cite{romik}) and some simplification can be used in order to obtain
\begin{eqnarray}\label{stirling}
{m^n(m-1)! \over n!\,(n+m-1)!}
\leq c~{m^n  e^{2n} \over n^{2n+m}}, \quad n\in\mathbb{N},
\end{eqnarray}
for  some positive constant $c$.
By combining \eqref{stirling} along with both \eqref{step3} and \eqref{step4} we 
finish the proof.
\end{proof}

\begin{proof}[{\bf Proof of Theorem \ref{maintheorem}}] Hypothesis  \eqref{normbound} implies  $\mathcal{K}_{0,r}:L^2(S^m)\to L^2(S^m)$ is bounded for all $r\in\mathbb{N}$. By Lemma \ref{lemmavalorsingular} we have
\begin{equation*}
s_{d_n^{m+1}}(\mathcal{K})\leq c~{e^{2n}m^{n} \over n^{2n+m}}~s_1(\mathcal{K}_{0,n}), \quad n\in\mathbb{N},
\end{equation*}
where $c>0$ does not depend on $n$. Furthermore, as Equality \eqref{normL2singvalues} implies $s_1(\mathcal{K}_{0,n})\leq ||K_{0,n}||_2$, we invoke  again Hypothesis  \eqref{normbound} and achieve
\begin{equation}\label{estimatives2n1}
s_{d_n^{m+1}}(\mathcal{K})\leq c~M~{e^{2n}m^{n}R^n \over n^{2n+m}},  \quad n\in\mathbb{N}.
\end{equation}

 From 
 \eqref{dnm} we 
derive the limit formula
\begin{equation}\label{limitformula}
\lim_{n\to \infty}\frac{d_n^{m+1}}{n^{m}} =\frac{2}{m!},
\end{equation}which in its turn implies
\begin{equation}\label{estimativednm}
d_n^{m+1}\leq (\delta n)^m,\quad n\gg1,
\end{equation}
where $\delta=2$ for $m=2$, and $\delta=1$ for $m\geq3$.
Since the singular values of the integral operator generated by $K$ are arranged in a non-increasing order
 then 
\begin{eqnarray}\label{step0}
s_{(\delta n)^m}(\mathcal{K})\leq c~M~{e^{2n}m^{n}R^n \over n^{2n+m}},  \quad n\gg1.
\end{eqnarray}

Now let $\alpha_n=o(\sqrt[m]{n})$,   as $n\to\infty$,  be a positive sequence and write $\alpha'_n:=\alpha_{\delta^m(n+1)^m}$, $n=1,2,\dots$. It follows from \eqref{step0} that
\begin{equation*}\label{compareseries}
\sum_{n\gg1}^{\infty}(2\delta)^{n+m\alpha'_n} n^{n+m(1+\alpha'_n)}s_{(\delta n)^m}(\mathcal{K})
\leq c~M \sum_{n\gg1}^{\infty}{e^{2n}m^{n}R^n(2\delta)^{n+m\alpha'_n} n^{m\alpha'_n}\over n^{n}} ,
\end{equation*}
where the series on the right side converges, as $\alpha_n=o(\sqrt[m]{n})$ implies $\alpha'_n=o(n)$. 
Hence, the series
\begin{equation}\label{compareseries1}
\sum_{n=1}^{\infty}(2\delta)^{n+m\alpha'_n} n^{n+m(1+\alpha'_n)}s_{(\delta n)^m}(\mathcal{K})
\end{equation}
is convergent.

In order to conclude the proof we will prove the convergence of the series
\begin{equation}\label{seriessingularvalues}
\sum_{n=1}^{\infty}n^{{\sqrt[m]{n}\over\delta m}+{\alpha_n}}s_{n}(\mathcal{K}).
\end{equation}
This will be done by bounding the very last series by \eqref{compareseries1}.  A reindexation of \eqref{seriessingularvalues} gives
\begin{equation*}
	\sum_{n=\delta^m+1}^{\infty}\hspace{-.3cm} n^{{\sqrt[m]{n}\over\delta m}+{\alpha_n}}s_{n}(\mathcal{K})
	= \sum_{n=1}^{\infty}\sum_{k=1}^{\delta^m(n+1)^m-(\delta n)^m}\hspace{-.8cm}[(\delta n)^m+k]^{\left({\sqrt[m]{(\delta n)^m+k}\over\delta m}+{\alpha_{(\delta n)^m+k}}\right)}s_{(\delta n)^m+k}(\mathcal{K}).
\end{equation*}
Since $k\leq\delta^m(n+1)^m-(\delta n)^m$,
\begin{eqnarray*}
	\sum_{n=\delta^m+1}^{\infty}\hspace{-.3cm} n^{{\sqrt[m]{n}\over\delta m}+{\alpha_n}}s_{n}(\mathcal{K})
	&\leq& \delta \sum_{n=1}^{\infty} \delta^{n+m\alpha'_n} (n+1)^{n+1+m\alpha'_n} \sum_{k=1}^{\delta^m(n+1)^m-(\delta n)^m}\hspace{-.8cm}s_{(\delta n)^m+k}(\mathcal{K}).
\end{eqnarray*}
The sequence $\{s_n(\mathcal{K})\}$ is non-increasing then
\begin{align*}
	\sum_{n=\delta^m+1}^{\infty}\hspace{-.3cm} n^{{\sqrt[m]{n}\over\delta m}+{\alpha_n}}& s_{n}(\mathcal{K})
	\leq 
	\delta \sum_{n=1}^{\infty} \delta^{n+m\alpha'_n} (n+1)^{n+1+m\alpha'_n}s_{(\delta n)^m}(\mathcal{K}) \sum_{k=1}^{\delta^m(n+1)^m-(\delta n)^m}\hspace{-.8cm} 1\\
	&\leq \delta \sum_{n=1}^{\infty} \delta^{n+m\alpha'_n} (n+1)^{n+1+m\alpha'_n}s_{(\delta n)^m}(\mathcal{K}) [\delta^m(n+1)^m-(\delta n)^m].
\end{align*}
\\
Moreover, from Lemma 4.5 of \cite{cas-men} we know that
$$
\delta^m(n+1)^{m}-(\delta n)^{m}\leq m\delta^m2^{m-1}n^{m-1},\quad n=1,2,\dots,
$$
and it is easy to see that 
$
(n+1)^{n+1+m\alpha'_n}\leq 2^{n+1+m\alpha'_n}n^{n+1+m\alpha'_n}$, $n=1,2,\dots$. Therefore, we obtain
\begin{eqnarray*}
\sum_{n=2^m+1}^{\infty} n^{{\sqrt[m]{n}\over2m}+{\alpha_n}}s_{n}(\mathcal{K})
 &\leq& m2^{m} \delta^{m+1}  \sum_{n=1}^{\infty} (2\delta)^{n+m\alpha'_n} n^{n+m(1+\alpha'_n)}s_{(\delta n)^m}(\mathcal{K}),
 \end{eqnarray*}
and by \eqref{compareseries1} the proof is complete.
\end{proof}

We now proceed to prove Theorem \ref{maintheorem2}. The proof follows very closely the last one we provided but it is important to notice some implications arising from the positivity assumption before starting it.

Since $K$ $L^2$-PD implies that the integral operator $\mathcal{K}$ is self-adjoint then it can be represented as (\cite[p.36]{schatten})
\begin{equation}\label{kernelseriesexpansion}
K(x,y)=\sum_{n=0}^{\infty}\sum_{k=1}^{d_n^m}a_{n,k}Y_{n,k}(x)\overline{Y_{n,k}(y)}, \quad x,y\in S^m.
\end{equation}
From Theorem 7.7 of \cite{men-pia} we obtain
 \begin{equation}
K_{0,r}(x,y) = \mathcal D_y^rK(x,y) = \sum_{n=1}^\infty\sum_{k=1}^{d_n^m} a_{n,k}Y_{n,k}(x)\mathcal D_y^r\overline{Y_{n,k}(y)},
\end{equation}
while \eqref{eigenfunctionharmonic} yields
\begin{equation*}
K_{0,r}(x,y) = \sum_{n=1}^\infty\sum_{k=1}^{d_n^m} \underbrace{a_{n,k}\left[\frac{n(n+m-1)}{m}\right]^r}_{\text{new coefficient: }  b_{n,k}^{m,r}}Y_{n,k}(x)\overline{Y_{n,k}(y)}.
\end{equation*}
Therefore, for $r$ a positive integer number, it follows that 
\begin{equation}\label{eq-valor-sing-mercer}
s_1(\mathcal K_{0,r}) = \max\{b_{1,k}^{m,r}: k=1,\ldots,d_1^m\}
 =\max\left\{a_{1,k}
 : 1\leq k\leq d_1^m=m+1\right\}
 := a_1^m,
\end{equation}
i.e., $s_1(\mathcal K_{0,r}) $ does not depend upon $r$.\\

Next, we provide the proof of Theorem  \ref{maintheorem2}.

\begin{proof}[{\bf Proof of Theorem \ref{maintheorem2}}] 
By Lemma \ref{lemmavalorsingular} we have
\begin{equation*}
s_{(\delta n)^{m}}(\mathcal{K})\leq c~{e^{2n}m^{n} \over n^{2n+m}}~s_1(\mathcal{K}_{0,n}), \quad n\gg1,
\end{equation*}
where $c>0$ does not depend on $n$.

Since $K$ is $L^2$-PD and belongs to $W_2^\infty$, it follows from \eqref{eq-valor-sing-mercer} that 
\begin{equation}\label{estimatives2n2}
s_{(\delta n)^{m}}(\mathcal{K})\leq ca_1^m~{e^{2n}m^{n} \over n^{2n+m}}, \quad n\gg1.
\end{equation}

From this point, we can follow the procedure of the previous proof from \eqref{step0}  and arrive at the desired convergent series after applying Equality \eqref{lemma sing val1}.
\end{proof}


\section{Optimality-type result} \label{sec_optimal}

In this section we provide an optimality-type result for Theorem \ref{maintheorem2}. We present a kernel $K$ on $S^m$, $m\geq2$, that  satisfies the hypotheses of Theorem \ref{maintheorem2} 
but  
\begin{equation}
\lambda_n(\mathcal{K})\neq o(n^{-\frac{\sqrt[m]{n}}{\delta m}-\beta_n}), \text{ as } n\to\infty,\quad
\delta=\begin{cases}
2,& \text{if } m=2,\\
1,& \text{if } m\geq3,
\end{cases}
\end{equation}
 for each positive sequence $\{\beta_n\}_n$  satisfying a certain property  that allows to conclude that $\beta_n \neq o(\sqrt[m]n)$, as $n\to\infty$. 
 Actually, Theorem \ref{optimal-G} in the sequel does not imply that the decay rate produced by Theorem \ref{maintheorem2} is the best possible, necessarily.\ But it does indicate both that it is a  good decay rate and a direction to be followed to obtain  optimality.

\begin{theorem}\label{optimal-G} Let $m\geq2$. There exists a kernel $K$ satisfying the conditions of Theorem \ref{maintheorem2} such that  
	\begin{equation}\label{serieseigenvalue}
	\lim_{n\to\infty}n^{\frac{\sqrt[m]{n}}{\delta m}+\beta_n}\lambda_n(\mathcal{K}),\quad \delta=\begin{cases}
	2,& \text{if } m=2,\\
	1,& \text{if } m\geq3,
	\end{cases}
	\end{equation}
	does not converge to zero, for all positive sequences $\{\beta_n\}_n$ satisfying
	\begin{equation}\label{eq-lim-beta_n}
	c\sqrt[m]n  < \beta_n, \quad n\gg1,
	\end{equation}
	where $c$ is a positive constant.
\end{theorem}

\begin{proof}  
	Consider the kernel $K$ having the condensed spherical harmonic expansion 
	\begin{equation}\label{kerneloptimal}K(x,y)\sim 1+\sum_{n=1}^{\infty} \frac{d_n^m}{m^nn^{m-1}}  P_n^m(x\cdot y),
	\quad x,y \in S^m,\end{equation}
	where $P_n^m$ is the Legendre polynomial of degree $n$ associated with the integer $m$.
	
	Since $|P_n^m(x\cdot y)|\leq1$, $x,y\in S^m$, and $d_n^m=O(n^{m-1})$, as $n\to\infty$, then there is $c_1>0$, independent of $n$, such that
	\begin{equation*}\left|\sum_{n=1}^{\infty} \frac{d_n^m}{m^nn^{m-1}}  P_n^m(x\cdot y)\right|\leq
	c_1\sum_{n=1}^{\infty} \frac1{m^n},\quad x,y\in S^m. \end{equation*}
	It follows that $K$ is positive definite on $S^m$ (\cite{scho-42}) and continuous on $S^m\times S^m$ as the series on the right side converges. Therefore, it is also $L^2$-PD by Theorem 2.1 of \cite{fe-men}.
	
The well-known addition formula (\cite{mori}) can be used to see that
	\begin{eqnarray*}
	K_{0,r}(x,y)\sim \frac{1}{m^r}\sum_{n=1}^{\infty}  \frac{d_n^m n^r(n+m-1)^r}{m^nn^{m-1}}   P_n^m(x\cdot y),\quad x,y \in
	S^m \mbox{ and } r=1,2,\dots.\end{eqnarray*}
The same reasoning we used in order to show the continuity of $K$ guarantees that $K_{0,r}$ is continuous on $S^m\times S^m$ for all $r\in\mathbb{N}$. 

Hence, $K\in W_2^{\infty}$ and  $\mathcal{K}_{0,r}$ is a bounded operator on $L^2(S^m)$ for all $r\in\mathbb{N}$. Therefore $K$ satisfies the hypotheses of Theorem \ref{maintheorem2} and then
$$
	\lambda_n(\mathcal{K})=o\left(n^{-{\frac{\sqrt[m]{n}}{\delta m} - \alpha_n}}\right),\quad \text{as $n\to\infty$,}	
$$
for all positive sequences $\alpha_n=o(\sqrt[m]{n})$, as $n\to\infty$.

	Now let  $\{\beta_n\}_n$ satisfy \eqref{eq-lim-beta_n}, for all $n\geq n_0$, and
        define $a_n:=n^{\frac{\sqrt[m]{n}}{\delta m}+\beta_n}\lambda_n(\mathcal{K})$, $n\in\N$. We will show that the subsequence $\{a_{d_n^{m+1}}\}_n$ diverges to $\infty$, as $n\to\infty$. Indeed, by \eqref{eq-lim-beta_n} 	
        \begin{align*}
			a_{d_n^{m+1}}=(d_n^{m+1})^{\frac{\sqrt[m]{d_n^{m+1}}}{\delta m}+\beta_{d_n^{m+1}}} \lambda_{d_n^{m+1}}(\mathcal{K})
			>
			(d_n^{m+1})^{\left(\frac1{\delta m}+c\right)\sqrt[m]{d_n^{m+1}}} \lambda_{d_n^{m+1}}(\mathcal{K}),\quad n\geq n_0.
		\end{align*}
		 By \eqref {kerneloptimal} and \eqref{index}, $\lambda_{d_n^{m+1}}(\mathcal{K})=1/{m^nn^{m-1}}$, then 
		 \begin{align*}
		 a_{d_n^{m+1}} >
		 (d_n^{m+1})^{\left(\frac1{\delta m}+c\right)\sqrt[m]{d_n^{m+1}}}\frac1{m^nn^{m-1}},\quad n\geq n_0.
		 \end{align*}
	It follows from \eqref{limitformula} that there is $n_1\geq \max\{n_0, m!\}$ such that
	$$
	\frac{n}{m!}<\left(\frac{n}{m!}\right)^m<d_n^{m+1}, \quad n\geq n_1. 	$$
Hence,
	\begin{align*}
		a_{d_n^{m+1}} >
		\left(\frac{n}{m!}\right)^{\left(\frac1{\delta m}+c\right)\sqrt[m]{\left(\frac{n}{m!}\right)^m}} \frac1{m^nn^{m-1}}
		=\left(\frac{n}{m!}\right)^{\frac{\left(\frac1{\delta m}+c\right)n}{m!}} \frac1{m^nn^{m-1}},\quad n\geq n_1,
	\end{align*}
	that is,
		\begin{align*}
			a_{d_n^{m+1}} > \frac{n^{c_2n-m+1}}{(m(m!)^{c_2})^n},\quad n\geq n_1,
\end{align*}			
where, 
$c_2:=\frac1{m!}\left(\frac1{\delta m}+c\right)
$ 
is a  positive constant.

This shows that the subsequence $\{a_{d_n^{m+1}}\}_n$ diverges to $\infty$, as $n\to\infty$.\ Therefore, $a_n$  does not converge to 0, and we conclude that $\lambda_n(\mathcal{K})$ is not $o(n^{-\frac{\sqrt[m]{n}}{\delta m}-\beta_n})$, as $n\to\infty$.	
\end{proof}


\section{Examples} \label{sec-examples}
	
In this section we present some examples of kernels fitting into Theorem \ref{maintheorem2}. The first class of examples to be presented was studied in \cite{aze-men-sharp}  where the authors obtained estimates more precise than ours but using a strong technical hypothesis (see Equation (3.3) therein).\ A weaker one is enough to guarantee that dot product kernels satisfy Theorem \ref{maintheorem2}.

\begin{example} {\bf(Dot product kernel)} 
{\rm Assume $m\geq 2$ and consider the dot product kernel
\begin{equation}\label{dotproductexp}
K(x,y)=\sum_{n=1}^{\infty}b_n(x\cdot y)^n, \quad x,y\in S^m,
\end{equation}
where $b_n>0$, $n\in\mathbb{N}$, and 
\begin{equation}\label{bncondition}
\lim_{n\to\infty}{b_{n+1}\over b_n}
\end{equation} converges to some number 

strictly
  smaller than 1.

It is not difficult to see that $K$ is positive definite on $S^{m}$ and continuous on $S^{m}\times S^{m}$. Then it is also $L^2$-PD. Moreover, the decay hypothesis on $b_n$ implies that $K$ is actually infinitely many times differentiable in the usual sense on the sphere. As so, it belongs to $W_{2}^{\infty}$ and $\mathcal{K}_{0,r}$ is bounded on $L^2(S^m)$. Therefore, kernels in this class  satisfy Theorem \ref{maintheorem2}.
}

\end{example}

The Gaussian kernel is an example of a kernel with an expansion as \eqref{dotproductexp} satisfying \eqref{bncondition}. It plays an important role in many branches of mathematics such as learning theory, radial basis functions and methods in kernel-based spaces, e.g. image and video colorization (\cite{hof-sch-smola, mi-ka-le}).

\begin{example} {\bf(Gaussian kernel)} {\rm For $r>0$ a fixed real number the Gaussian kernel
$$
K(x,y)=e^{r/2}\exp\left(-\frac{\|x-y\|^2}{r}\right) , \quad x,y\in S^m,
$$
 can be represented as a dot product kernel as follows
$$
K(x,y)=\sum_{n=0}^{\infty}\frac{2^{n}}{n! r^n}\,(x\cdot y)^n, \quad x,y\in S^m.
$$
Its easy to see that $b_n=\frac{2^{n}}{n! r^n}$ satisfies \eqref{bncondition}. Then the eigenvalues $\lambda_n$ of the integral operator generated by the Gaussiann kernel satisfy \eqref{serieseigenvalues} and, consequently, $\lambda_n = o\left(n^{-{\sqrt[m]{n}/\delta m}-\alpha_n}\right)$, as $n\to\infty$. 
 
We observe that this particular behavior can be ratified by Theorem 2 in \cite{mi-ni-yao}.}
\end{example}

Next, we consider kernels that provide some of the functions most used in spherical models that arise in several areas such as geostatistics, mathematical physics, and probability theory (see \cite{arafat-porcu-bev-mat} and references therein).

\begin{example} {\bf(Multiquadric kernel)} {\rm  For $\sigma>0$ and $0<\delta<1$ the multiquadric kernel
\begin{equation}\label{eq-multiquad-funct-S_d}
K(x,y) = \sigma^2\frac{(1-\delta)^{m-1}}{(1+\delta^2-2\delta\, x\cdot y)^{m-1\over 2}}, \quad x,y\in S^m,
\end{equation}
is positive definite on $S^m$ and has an expansion as (\cite{mo-porcu})
\begin{equation*}\label{eq-multiquad-coef-S_d}
K(x,y) = \sum_{n=0}^{\infty} \sigma^2(1-\delta)^{m-1}\binom{m+n-2}{n}\delta^n P_n^m(x\cdot y).
\end{equation*}
Clearly, it is continuous and then $L^2$-PD.

Furthermore, $K$ belongs to $W_2^{\infty}$ because 
$$K_{0,r}(x,y)=\sum_{n=1}^{\infty} \sigma^2(1-\delta)^{m-1}\binom{m+n-2}{n}\delta^n {n^r(n+m-2)^r\over m^r} P_n^m(x\cdot y)$$
is continuous for all $r\in\mathbb{N}$. This also implies that $\mathcal{K}_{0,r}:L^2(S^m)\to L^2(S^m)$ is bounded.}
\end{example}

\begin{example} {\bf(Spectral model kernel: M\o{}ller family)} {\rm  For $\alpha,\beta,\tau, \sigma>0$ the spectral model kernel
$$
K(x,y) = \sum_{n=0}^{\infty} {\sigma^2\over 1+\beta \exp({({n\over\alpha})^{\tau})}} P_n^m(x\cdot y)
$$
is positive definite on $S^m$ (\cite{mo-porcu}) and also satisfies the conditions of Theorem \ref{maintheorem2}.
}
\end{example}

\section{Appendix} \label{appendix}

In this section we summarize part of the Harmonic Analysis on compact two-point homogeneous spaces in order to be clear on how to obtain Theorems \ref{maintheorem} and \ref{maintheorem2} on this wider setting as it was mentioned in Remark \ref{rem-homog}. The reader can find more details in \cite{dai, kushpel2, platonov1}.

A compact two-point homogeneous space $\mathbb{M}^m$  of dimension $m\geq 1$ is both a Riemannian $m$-manifold and a compact symmetric space of rank $1$.\ According to Wang in \cite{wang}, they are:

- the unit spheres $\mathbb{S}^m$, $m=1,2,\dots$; 

- the real projective spaces $\mathbb{P}^m(\mathbb{R})$, $m=2,3,\dots$; 

- the complex projective spaces $\mathbb{P}^m(\mathbb{C})$, $m=4,6,\dots$; 

- the quaternion projective spaces $\mathbb{P}^m(\mathbb{H})$, $m=8,12,\dots$;

- $16$-dimensional Cayley's elliptic plane $\mathbb{P}^{16}$.

Each $\mathbb{M}^m$ has an invariant Riemannian (geodesic) normalized metric $d(\cdot,\cdot)$ such that all geodesics on $\mathbb{M}^m$ are closed and have length $2\pi$. Moreover, $\mathbb{M}^m$ can be endowed with the measure $d\sigma(x)$ induced by a normalized left Haar measure. 

Let $L^2(\mathbb{M}^m):=L^2(\mathbb{M}^m,\sigma)$ be the Hilbert space consisting of all square integrable functions $f:\mathbb{M}^m\to \mathbb C$ with the norm $\|\cdot\|_2$ induced by the inner product
\begin{eqnarray}\label{eq-prod-int-M}
\langle f, g \rangle_2=\int_{\mathbb{M}^m}f(x)\overline{g(x)}d\sigma(x), \quad f,g\in L^2(\mathbb{M}^m),
\end{eqnarray}
and  $\mathcal{B}$ the Laplace-Beltrami operator on $\mathbb{M}^m$.\ It is known that  the spectrum of $\mathcal{B}$ is discrete, real and non-positive.  
Indeed,  in geodesic polar coordinates we can write $\mathcal B = \mathcal B_\theta + \mathcal B'$, where $\mathcal B'$ denotes the Laplace-Beltrami operator on the sphere in $\mathbb{M}^m$ of radius $\theta$ (\cite[p. 420]{dai})  and $\mathcal B_\theta$ is its radial part, which after the change of variables $x=\cos(2\lambda\theta)$ can be written as (\cite[Eq. (5)]{kushpel2})
$$
\mathcal B_x =(1-x)^{-\alpha}(1+x)^{-\beta}\frac{d}{dx}(1-x)^{1+\alpha}(1+x)^{1+\beta}\frac{d}{dx},
$$
where $\alpha:=(\sigma+\rho-1)/2$, $\beta:=(\rho-1)/2$ and for 
$$
\begin{array}{llll}
S^m, \, m=1,2,\dots:& \sigma=0, &\rho=m-1,& \lambda=1/2; \\
\mathbb{P}^m(\mathbb{R}), m=2,3,\dots: &\sigma=m-1, &\rho=0, &\lambda=1/4; \\
\mathbb{P}^m(\mathbb{C}), m=4,6,\dots: & \sigma=m-2, &\rho=1, &\lambda=1/2; \\
\mathbb{P}^m(\mathbb{H}), m=8,12,\dots: & \sigma=m-4,&\rho=3, &\lambda=1/2; \\
\mathbb{P}^{16}: & \sigma=8, & \rho=7, &\lambda=1/2.
\end{array}
$$
Observe that $\alpha=(m-2)/2$ for all spaces $\mathbb{M}^m$. It is not difficult to see that the eigenfunctions of $\mathcal B_x$ are the well known Jacobi polynomials $P_k^{(\alpha,\beta)}(x)$ and the corresponding eigenvalues are $-k(k+\alpha+\beta+1)$. Details concerning to Jacobi polynomials can be found in \cite{szego}. 

We consider $m\geq2$ and the Laplace-Beltrami operator  $\Delta := -\mathcal B_x$   so that its eigenvalues can be arranged in an increasing order and given by $\{1\}\cup\{k(k+\alpha+\beta+1): k=1,2,\ldots\}$. 

Each  eigenvalue is associated to an eigenspace $\mathcal{H}^m_k:= \mathcal{H}^m_k(\mathbb M^m)$ of $\Delta$. They are mutually orthogonal with respect to the inner product \eqref{eq-prod-int-M} and 
$$
L^2(\mathbb{M}^m) = \bigoplus_{k=0}^\infty \mathcal{H}^m_k.
$$

\begin{remark} It is known that functions on  $\mathbb{P}^m(\mathbb{R})$ can be seen as even functions on $\mathbb{S}^m$.\ Then $L^2(\mathbb{P}^m(\mathbb{R}))$ can be identified with $\oplus_n\mathcal{H}_{2n}^m(\mathbb{S}^m)$. Thus the decay rates in the case  $\mathbb{M}^m=\mathbb{P}^m(\mathbb{R})$ can be obtained directly from the case $\mathbb{M}^m=\mathbb{S}^m$. From now on we consider $\mathbb{M}^m$ as being $S^m, \mathbb{P}^m(\mathbb{C}),\mathbb{P}^m(\mathbb{H}),$ or $\mathbb{P}^{16}$.
\end{remark}

The dimensions $d_k^m:=\dim \mathcal{H}^m_k$ are given by (\cite[Eq. (7)]{kushpel2})
$$
d_k^m = \frac{\Gamma(\beta+1)(2k+\alpha+\beta+1)\Gamma(k+\alpha+1)\Gamma(k+\alpha+\beta+1)}{\Gamma(\alpha+1)\Gamma(\alpha+\beta+2)\Gamma(k+1)\Gamma(k+\beta+1)},
$$
and satisfy
(\cite[p.405]{dai})
$$
d_k^m \asymp k^{m-1},\quad \text{ as } k\to\infty.
$$
If we write $\{Y_{k,j}: j=1,2, \ldots, d_k^m\}$ for an orthonormal basis of $\mathcal{H}^m_k$, then $\{Y_{k,j}: j=1,2, \ldots, d_k^m; k=0,1,\dots\}$ is an orthonormal basis of $L^2(\mathbb{M}^m)$.\ This permits us to consider naturally Fourier expansions on $L^2(\mathbb{M}^m)$:
$$
f=\sum_{k=0}^\infty\sum_{j=1}^{d_k^m}\langle f,Y_{k,j}\rangle_2Y_{k,j}.
$$

Given $r>0$ we define the $r$-th order Laplace-Beltrami operator $\Delta^r$ on $\mathbb{M}^m$ in the distributional sense given by 
$$
\Delta^{r}(f)\sim\sum_{k=1}^{\infty}(k(k+\alpha+\beta+1))^r\,\mathcal{Y}_k(f), \quad f\in L^2(\mathbb{M}^m),
$$
where  $\mathcal{Y}_k$ is the orthogonal projection of $L^2(\mathbb{M}^m)$ onto $\mathcal{H}^m_k$, $k=0,1,\ldots$. Hence (\cite{dai})
$$
\Delta^r(Y) = k^r(k+\alpha+\beta+1)^r Y, \quad Y\in \mathcal{H}^m_k.
$$
We can now define  Sobolev classes of Laplace differentiable functions by
$$
W_2^r(\mathbb{M}^m):=\left\{f\in L^2(\mathbb{M}^m):\Delta^r(f)\in L^2(\mathbb{M}^m)\right\},
$$ 
as well as the space of infinitely many times Laplace-Beltrami differentiable functions on $\mathbb{M}^m$ as 
$$
W_2^\infty(\mathbb{M}^m):= \bigcap_r W_2^r(\mathbb{M}^m)=\left\{f\in L^2(\mathbb{M}^m):\Delta^r(f)\in L^2(\mathbb{M}^m),\, r=1,2,\ldots\right\}.
$$

 In this setting the analogous operator to the Laplace-Beltrami integral operator stated in \eqref{eq3.1} is the operator $\mathcal J^r$ defined by (see \cite[p.10]{schatten}) 
$$
\mathcal J^r(f) = \langle f,Y_{0,1}\rangle_2 +\sum_{k=1}^\infty\sum_{j=1}^{d_k^m}k^{-r}(k+\alpha+\beta-1)^{-r}\langle f,Y_{k,j}\rangle_2Y_{k,j}, \quad f\in L^2(\mathbb{M}^m), \, r>0.
$$
For each positive integer $r$, the operator $\mathcal J^r:L^2(\mathbb{M}^m)\to L^2(\mathbb{M}^m)$ is compact and satisfies $\Delta^r \mathcal J^r f = f$, for $f\in L^2(\mathbb{M}^m)$. Moreover  the singular values  $s_n(\mathcal J^r)$ of $\mathcal J^r$ can be arranged in blocks so that the first one contains the singular value $s_0(\mathcal J^r)=1$ and the $(n+1)$-th block, $n\geq 1$, contains $d_n^m$ elements equal to $n^{-r}(n+\alpha+\beta-1)^{-r}$.\ Consequently 
$$
d_0^m+d_1^m+\cdots +d_{n-1}^m+d_n^m=\tau_n^{m},
$$
where $\tau_n^{m}$ is the dimension of $\bigoplus_{k=0}^n \mathcal{H}^m_k$.\  A result analogous to Lemma \ref{lemmafactorization} can be stated and proved following the ideas of \cite{cas-men}.

Likewise in the spherical case, we consider kernels $K$ on $\mathbb{M}^m$ and the integral operator
\begin{equation*}
\mathcal{K}(f)(x)=\int_{\mathbb{M}^m}K(x,y)f(y)d\sigma(y), \quad f\in L^2(\mathbb{M}^m), \, x\in\mathbb{M}^m.
\end{equation*}
The action of the $r$-th order Laplace-Beltrami operator on the kernel $K$ with respect to variable $y$ can be written as in Section \ref{sec_main-results} by
\begin{equation*}
K_{0,r}:= \Delta_y^rK, \quad r=1,2,\ldots,
\end{equation*}
while $\mathcal{K}_{0,r}$ denotes the integral operator generated.   Finally, we say that $K$ {\em belongs to} $W_2^{\infty}(\mathbb{M}^m)$ whenever 
\begin{equation*}
K(x,\cdot)\in W_2^{\infty}(\mathbb{M}^m), \quad x\in \mathbb{M}^m\quad \text{a.e.}.
\end{equation*}

At this point it is clear how to prove versions of Theorems \ref{maintheorem} and \ref{maintheorem2} on this context:
it is enough to repeat with minor adaptations the steps presented in Section \ref{sec_proofmain}.

\section{Acknowledgement}

We thank the anonymous referees for their valuable suggestions.


\bibliographystyle{amsalpha}

\begin{thebibliography}{MNPR18}
	
	\bibitem[AM14]{aze-men-sharp}
	D.~Azevedo and V.~A. Menegatto, \emph{Sharp estimates for eigenvalues of
		integral operators generated by dot product kernels on the sphere}, J.
	Approx. Theory \textbf{177} (2014), 57--68. 
	
	\bibitem[APBM18]{arafat-porcu-bev-mat}
	A.~Arafat, E.~Porcu, M.~Bevilacqua, and J.~Mateu, \emph{Equivalence and
		orthogonality of {G}aussian measures on spheres}, J. Multivariate Anal.
	\textbf{167} (2018), 306--318.
	
	\bibitem[BD05]{dai}
	G.~Brown and F.~Dai, \emph{Approximation of smooth functions on compact
		two-point homogeneous spaces}, J. Funct. Anal. \textbf{220} (2005), no.~2,
	401--423.
	
	\bibitem[BS77]{bir-sol}
	M.~{\u{S}}. Birman and M.~Z. Solomjak, \emph{Estimates for the singular numbers
		of integral operators}, Russian Math. Surveys \textbf{32} (1977), no.~1,
	15--89.
	
	\bibitem[CH99]{chang-ha99}
	C-H Chang and C-W Ha, \emph{On eigenvalues of differentiable positive definite
		kernels}, Integral Equations Operator Theory \textbf{33} (1999), no.~1, 1--7.

	
	\bibitem[CM12]{cas-men}
	M.~H. Castro and V.~A. Menegatto, \emph{Eigenvalue decay of positive integral
		operators on the sphere}, Math. Comp. \textbf{81} (2012), no.~280,
	2303--2317. 
	
	\bibitem[CMO13]{ca-men-oliv}
	M.~H. Castro, V.~A. Menegatto, and C.~P. Oliveira, \emph{Laplace-{B}eltrami
		differentiability of positive definite kernels on the sphere}, Acta Math.
	Sin. (Engl. Ser.) \textbf{29} (2013), no.~1, 93--104.
	
	\bibitem[FM09]{fe-men}
	J.~C. Ferreira and V.~A. Menegatto, \emph{Eigenvalues of integral operators
		defined by smooth positive definite kernels}, Integral Equations Operator
	Theory \textbf{64} (2009), no.~1, 61--81. 
	
	\bibitem[FMP08]{fe-men-pe}
	J.~C. Ferreira, V.~A. Menegatto, and A.~P. Peron, \emph{Integral operators on
		the sphere generated by positive definite smooth kernels}, J. Complexity
	\textbf{24} (2008), no.~5-6, 632--647.
	
	\bibitem[Fre03]{fred}
	I.~Fredholm, \emph{Sur une classe d'\'{e}quations fonctionnelles}, Acta Math.
	\textbf{27} (1903), no.~1, 365--390. 
	
	\bibitem[Han90]{han90}
	Y.~B. Han, \emph{Positive definite kernels in the class {$H^p$} and their
		eigenvalues}, Acta Math. Sci. (Chinese) \textbf{10} (1990), no.~2, 126--131.

	
	\bibitem[HSS08]{hof-sch-smola}
	T.~Hofmann, B.~Sch\"{o}lkopf, and A.~J. Smola, \emph{Kernel methods in machine learning}, Ann. Statist. \textbf{36} (2008), no.~3, 1171--1220. 
	\bibitem[JMS14]{jor-men-sun}
	T.~Jord{\~a}o, V.~A. Menegatto, and X.~Sun, \emph{Eigenvalue sequences of
		positive integral operators and moduli of smoothness}, Approximation theory
	{XIV}: {S}an {A}ntonio 2013, Springer Proc. Math. Stat., vol.~83, Springer,
	Cham, 2014, pp.~239--254. 
	
	\bibitem[K{\"o}n86]{konig}
	H.~K{\"o}nig, \emph{Eigenvalue distribution of compact operators}, Operator
	Theory: Advances and Applications, vol.~16, Birkh\"auser Verlag, Basel, 1986.

	
	\bibitem[Kot78]{kotljar}
	B.~D. Kotljar, \emph{Singular numbers of integral operators}, Differentsial'
	nye Uravneniya \textbf{14} (1978), no.~8, 1473--1477, 1532. 
	
	\bibitem[KT12]{kushpel2}
	A.~Kushpel and S.~A. Tozoni, \emph{Entropy and widths of multiplier operators
		on two-point homogeneous spaces}, Constr. Approx. \textbf{35} (2012), no.~2,
	137--180.
	
	
	\bibitem[LM72]{lions-ma}
	J.-L. Lions and E.~Magenes, \emph{Non-homogeneous boundary value problems and
		applications. {V}ol. {I}}, Springer-Verlag, New York-Heidelberg, 1972,
	Translated from the French by P. Kenneth, Die Grundlehren der mathematischen
	Wissenschaften, Band 181.
	
	\bibitem[LR84]{little-reade}
	G.~Little and J.~B. Reade, \emph{Eigenvalues of analytic kernels}, SIAM J.
	Math. Anal. \textbf{15} (1984), no.~1, 133--136.
	
	\bibitem[MKL10]{mi-ka-le}
	H.~Q. Minh, S.~H. Kang, and T.~M. Le, \emph{Image and video colorization using
		vector-valued reproducing kernel {H}ilbert spaces}, J. Math. Imaging Vision
	\textbf{37} (2010), no.~1, 49--65. 
	
	\bibitem[MNPR18]{mo-porcu}
	J.~M{\o}ller, M.~Nielsen, E.~Porcu, and E.~Rubak, \emph{Determinantal point
		process models on the sphere}, Bernoulli \textbf{24} (2018), no.~2,
	1171--1201. 
	
	\bibitem[MNY06]{mi-ni-yao}
	H.~Q. Minh, P.~Niyogi, and Y.~Yao, \emph{Mercer's theorem, feature maps, and
		smoothing}, Learning theory, Lecture Notes in Comput. Sci., vol. 4005,
	Springer, Berlin, 2006, pp.~154--168.
	
	\bibitem[Mor98]{mori}
	M.~Morimoto, \emph{Analytic functionals on the sphere}, Translations of
	Mathematical Monographs, vol. 178, American Mathematical Society, Providence,
	RI, 1998.
	
	\bibitem[MP11]{men-pia}
	V.~A. Menegatto and A.~C. Piantella, \emph{Old and new on the
		{L}aplace-{B}eltrami derivative}, Numer. Funct. Anal. Optim. \textbf{32}
	(2011), no.~3, 309--341. 
	
	\bibitem[Par92]{parfenov}
	O.~G. Parf\"{e}nov, \emph{Estimates for the singular numbers of integral
		operators with analytic kernels}, Vestnik S.-Peterburg. Univ. Mat. Mekh.
	Astronom. (1992), no.~vyp. 2, 24--32, 114.
	
	\bibitem[Pie87]{pietsch}
	A.~Pietsch, \emph{Eigenvalues and {$s$}-numbers}, Cambridge Studies in Advanced
	Mathematics, vol.~13, Cambridge University Press, Cambridge, 1987.

	
	\bibitem[Pla09]{platonov1}
	S.~S. Platonov, \emph{On some problems in the theory of the approximation of
		functions on compact homogeneous manifolds}, Mat. Sb. \textbf{200} (2009),
	no.~6, 67--108. 
	
	\bibitem[Rea92]{reade}
	J.~B. Reade, \emph{Eigenvalues of smooth positive definite kernels}, Proc.
	Edinburgh Math. Soc. (2) \textbf{35} (1992), no.~1, 41--45. 
	
	\bibitem[Rom00]{romik}
	D.~Romik, \emph{Stirling's approximation for {$n!$}: the ultimate short
		proof?}, Amer. Math. Monthly \textbf{107} (2000), no.~6, 556--557.

	
	\bibitem[RR94]{raman-rao}
	S.~G. Raman and R.~V. Rao, \emph{Eigenvalues of integral operators on
		{$L_2(I)$} given by analytic kernels}, Integral Equations Operator Theory
	\textbf{18} (1994), no.~1, 109--117.
	
	\bibitem[Sch42]{scho-42}
	I.~J. Schoenberg, \emph{Positive definite functions on spheres}, Duke Math. J.
	\textbf{9} (1942), 96--108. 
	
	\bibitem[Sch70]{schatten}
	R.~Schatten, \emph{Norm ideals of completely continuous operators}, Ergebnisse
	der Mathematik und ihrer Grenzgebiete - Band 27, Springer-Verlag, Berlin
	Heidelberg GmbH, 1970.
	
	\bibitem[Sze59]{szego}
	G.~Szeg{\"o}, \emph{Orthogonal polynomials}, American Mathematical Society
	Colloquium Publications, Vol. 23. Revised ed, American Mathematical Society,
	Providence, R.I., 1959. 
	
	\bibitem[Wan52]{wang}
	H-C Wang, \emph{Two-point homogeneous spaces}, Ann. of Math. (2) \textbf{55}
	(1952), 177--191.
	
	\bibitem[WZ15]{wa-zhu}
	A.~J. Wathen and S.~Zhu, \emph{On spectral distribution of kernel matrices
		related to radial basis functions}, Numer. Algorithms \textbf{70} (2015),
	no.~4, 709--726.
	
\end{thebibliography}

\end{document}